\documentclass[preprint,11pt,3p]{elsarticle}
\usepackage{amsmath,amsthm,amsfonts,amssymb,arydshln}
\usepackage{pgfplots}
\pgfplotsset{compat=1.4}
\usepackage{xcolor}
\usepackage{algorithm}
\usepackage{multirow}
\usepackage{multicol}
\usepackage{booktabs}
\usepackage{algorithmic}
\usepackage{comment}
\usepackage{threeparttable}
\usepackage{tikz}
\usepackage{physics}
\usepackage{graphicx}
\usepackage{float}
\usepackage{enumitem}
\usetikzlibrary{arrows.meta, positioning}

\usepackage{mathrsfs}
\usepackage{sectsty}
\definecolor{astral}{RGB}{46,116,181}
\sectionfont{\color{astral}}
\newtheorem{theorem}{Theorem}[section]
\newtheorem{lemma}[theorem]{Lemma}
\newtheorem{corollary}[theorem]{Corollary}

\newtheorem{definition}[theorem]{Definition}
\newtheorem{example}[theorem]{Example}

\usepackage{hyperref}
\usepackage{subfigure}
\usepackage{scalerel}
\usepackage{stackengine,wasysym}

\usepackage{tikz}
\pgfdeclarelayer{bg} 
\pgfsetlayers{bg,main}

\usepackage{tikz,xcolor}
\definecolor{lime}{HTML}{A6CE39}
\definecolor{lightblue}{rgb}{0.0, 0.0, 0.5}
\definecolor{theoremblue}{RGB}{44, 130, 201}
\definecolor{directorange}{RGB}{255, 102, 0}
\DeclareRobustCommand{\orcidicon}{%
	\begin{tikzpicture}
	\draw[lime, fill=lime] (0,0)
	circle [radius=0.16]
	node[white] {{\fontfamily{qag}\selectfont \tiny ID}};
	\draw[white, fill=white] (-0.0625,0.095)
	circle [radius=0.007];
	\end{tikzpicture}
	\hspace{-2mm}
}
\foreach \x in {A, ..., Z}{%
	\expandafter\xdef\csname orcid\x\endcsname{\noexpand\href{https://orcid.org/\csname orcidauthor\x\endcsname}{\noexpand\orcidicon}}
}

\begin{document}
\begin{frontmatter}
\title{Additive results of the dual Moore--Penrose generalized inverse with an application to rigid-body motion estimation}
\author[a]{Yue Zhao}
\ead{yuezhao0303@163.com}

\author[b]{Daochang Zhang\corref{cor1}}
\ead{daochangzhang@126.com}
\cortext[cor1]{Corresponding author.}

\address[a]{Department of Math \& Stat, Georgia State University, Atlanta, GA 30303, USA}
\address[b]{College of Mathematics, Changchun Normal University, Changchun 130032, P.R. China.}

\begin{abstract}
We derive explicit formulas for the dual Moore--Penrose generalized inverse (DMPGI) of dual complex matrix sums under some conditions. We first express the DMPGI of a matrix partitioned into column blocks in terms of these blocks and associated projectors, and then derive two formulas for matrix sums.  Two examples illustrate the results. We further apply the formulas to rigid-body motion estimation with two observation points subject to a first-order perturbation of their positions that preserves the rank-deficient observation geometry. The DMPGI recovers the observable motion components from the velocity data.
\end{abstract}

\end{frontmatter}

\section{Introduction}

A dual matrix $\widehat A=A+\varepsilon A_0$, where $\varepsilon^2=0$, encodes a matrix and its first-order variation. Dual matrix methods have been used in rigid-body kinematics, the analysis of spatial mechanisms, and robotics \cite{Spatial,Rigidbody,Rigidbody2,Kinematic}.
In these applications, motion determination and mechanism synthesis often lead to systems of linear dual equations with rectangular or singular coefficient matrices. Generalized inverses are therefore useful for studying
the solvability of these systems and expressing their solutions \cite{DFPU2018,PVD2018,Udwadia2021}.

Research on dual generalized inverses has developed through both algebraic investigations and kinematic applications. Pennestr{\`i}, Valentini, and de Falco \cite{PVD2018} studied computational methods for dual generalized inverses and applied them to the synthesis of spatial linkages. De Falco, Pennestr{\`i}, and Udwadia \cite{DFPU2018} examined the Penrose conditions satisfied by dual generalized inverse formulas and considered applications to rigid-body kinematics. Udwadia \cite{Udwadia2021} further investigated the use of dual generalized inverses in linear dual equations, including least-squares and minimum-norm formulations. 

The dual Moore--Penrose generalized inverse (DMPGI) is defined by the four Penrose equations, but does not exist for every dual matrix \cite{DOall2020}. Its existence and representations have been studied
in \cite{defMP}, and its connection with the compact dual singular value decomposition has been examined in \cite{Brain}. Alternative extensions of the Moore--Penrose inverse have also been developed to include dual matrices that do not admit a DMPGI \cite{CUIQI2025}.
The present paper concerns the inverse defined by the four Penrose equations. 

Explicit block representations of the DMPGI are useful for solving linear dual equations with partitioned coefficient matrices. Cline \cite{SIAMTHM2} established representations for the generalized inverse of matrix sums in $\mathbb C$. We first derive a representation of
the DMPGI for column-partitioned dual complex matrices and then use it to extend the corresponding formulas for matrix sums under specified existence and invertibility conditions. We apply these formulas to rigid-body motion estimation with perturbed observation points, illustrating the recovery of observable motion components.

The paper is organized as follows.
Section~2 introduces the notation, reviews the existence condition of the DMPGI, and establishes auxiliary results. Section~3 develops a representation for $[\widehat U\ \widehat V]^\dagger$. This representation yields a formula for $(\widehat U\widehat U^*+\widehat V\widehat V^*)^\dagger$. Under the condition $\widehat U\widehat V^*=O$, we also obtain a formula for
$(\widehat U+\widehat V)^\dagger$. Section~4 illustrates these formulas with two explicit examples. Section~5 applies the formulas to rigid-body motion estimation from velocity measurements at two observation points with perturbed positions and explains why the angular velocity component along the line joining these points remains undetermined.

\section{Preliminaries and auxiliary results}
This section introduces the notation and preliminary results used in the sequel. We first recall the basic properties of dual complex matrices and the DMPGI, and then establish the auxiliary identities needed for the main results.
\subsection{Dual complex matrices and the DMPGI}
Throughout this paper, $\mathbb C$ denotes the complex field. We denote the identity matrix of order $n$ by $I_n$ and a zero matrix of appropriate dimensions by $O$. The symbols $(\cdot)^T$ and $(\cdot)^*$ denote transpose and conjugate transpose, respectively.

The algebra of dual numbers over $\mathbb C$ is
$\mathbb {DC}=
\{a+\varepsilon a_0:a,a_0\in\mathbb C\},$
where the dual unit $\varepsilon$ satisfies
$\varepsilon\neq 0$, $\varepsilon^2=0$, and commutes with complex scalars. For $\widehat a=a+\varepsilon a_0$, the coefficients $a$ and $a_0$ are called its standard and infinitesimal parts, respectively. Two dual complex numbers are equal if and only if their standard and infinitesimal parts are equal. Addition and multiplication are defined by
\[
\begin{aligned}
(a+\varepsilon a_0)+(b+\varepsilon b_0)
=(a+b)+\varepsilon(a_0+b_0),\qquad
(a+\varepsilon a_0)(b+\varepsilon b_0)
=ab+\varepsilon(ab_0+a_0b).
\end{aligned}
\]
Conjugation is defined componentwise by
$\widehat a^*=a^*+\varepsilon a_0^*$.
In particular, $\varepsilon^*=\varepsilon$.
The number $\widehat a$ is invertible if and only if
$a\neq 0$, in which case
\[
\widehat a^{-1}
=
\frac{1}{a}-\varepsilon\frac{a_0}{a^2}.
\]

Following \cite{MiaoHuang2023}, for a dual vector
$\widehat x=x+\varepsilon x_0\in\mathbb{DC}^n$,
the dual Euclidean norm is defined by
\[
\|\widehat x\|=
\begin{cases}
\displaystyle
\|x\|_2+\varepsilon
\frac{\operatorname{Re}(x^*x_0)}{\|x\|_2},
& x\neq 0,\\[2mm]
\varepsilon\|x_0\|_2,
& x=0,
\end{cases}
\]
where $\|\cdot\|_2$ denotes the ordinary Euclidean norm and $\operatorname{Re}(\cdot)$ denotes the real part of a complex number.

An $m\times n$ dual complex matrix is written as
\[
\widehat X=X+\varepsilon X_0,
\qquad
X,X_0\in\mathbb C^{m\times n},
\]
where $X$ and $X_0$ are called the standard and infinitesimal parts of $\widehat X$, respectively. The set of all $m\times n$ dual complex matrices is denoted by 
$\mathbb{DC}^{m\times n}$. A dual matrix $\widehat X$ is called appreciable if  $X\neq O$, and infinitesimal if $X=O$. Its range and null space are defined by
\[
\mathcal R(\widehat X)
=
\{\widehat X\widehat x:\widehat x\in\mathbb D^n\},
\qquad
\mathcal N(\widehat X)
=
\{\widehat x\in\mathbb D^n:\widehat X\widehat x=O\}.
\]

Subsequently, we recall several basic definitions and properties \cite{Brain}. The transpose and conjugate transpose of $\widehat X$ are
$\widehat X^T=X^T+\varepsilon X_0^T$ and
$\widehat X^*=X^*+\varepsilon X_0^*$, respectively. For matrices $\widehat X$ and $\widehat Y$ of compatible dimensions, $(\widehat X\widehat Y)^T=\widehat Y^T\widehat X^T$ and $(\widehat X\widehat Y)^*=\widehat Y^*\widehat X^*.$ A matrix $\widehat C=C+\varepsilon C_0\in\mathbb{DC}^{n\times n}$ is called invertible (or nonsingular) if there exists
$\widehat D\in\mathbb{DC}^{n\times n}$ such that
$\widehat C\widehat D=\widehat D\widehat C=I_n$.
The inverse is unique and is denoted by
$\widehat C^{-1}$. The matrix $\widehat C$ is invertible if and only if $C$ is invertible, in which case $\widehat C^{-1}=
C^{-1}-\varepsilon C^{-1}C_0C^{-1}.$ A dual matrix is said to have full column rank if its standard part has full column rank \cite{Wangfullrank}. Full row rank is defined analogously.

We next recall the definition of the DMPGI \cite{defMP}. Let
$\widehat G=G+\varepsilon G_0
\in\mathbb{DC}^{m\times n}$. A matrix $\widehat Y\in\mathbb{DC}^{n\times m}$ is called the DMPGI of $\widehat G$ if
\begin{equation}\label{defdmpgi}
\begin{aligned}
\widehat G\widehat Y\widehat G=\widehat G,\quad
\widehat Y\widehat G\widehat Y=\widehat Y,\quad
(\widehat G\widehat Y)^*=\widehat G\widehat Y,\quad
(\widehat Y\widehat G)^*=\widehat Y\widehat G.
\end{aligned}
\end{equation}
When it exists, the DMPGI is unique and is denoted by $\widehat G^\dagger$. Specifically, \(\widehat G^\dagger\) exists if and only if
\begin{align}\label{compatibility condition}
(I_m-GG^\dagger)G_0(I_n-G^\dagger G)=O,
\end{align}
where \(G^\dagger\) is the ordinary Moore–Penrose inverse of \(G\) \cite{defMP}. The compatibility condition \eqref{compatibility condition} is automatically satisfied when the standard part $G$ has full column or full row rank. Whenever this condition holds, the DMPGI is given by
\[
\widehat G^\dagger=G^\dagger-\varepsilon R,
\]
where
$
R=G^\dagger G_0G^\dagger
-(G^*G)^\dagger G_0^*(I_m-GG^\dagger)-(I_n-G^\dagger G)G_0^*(GG^*)^\dagger.
$

For convenience, we denote the class of dual complex matrices admitting a DMPGI by
\[
\mathbb{DC}_z^{m\times n}
=
\left\{
G+\varepsilon G_0\in\mathbb{DC}^{m\times n}
\;\middle|\;
(I_m-GG^\dagger)G_0(I_n-G^\dagger G)=O
\right\}.
\]
Equivalently, $\mathbb{DC}_z^{m\times n}$ comprises all $m\times n$ dual complex matrices for which the DMPGI exists.

The existence of the DMPGI is closely connected with the compact dual singular value decomposition. The following definition is distilled from the characterization of the compact dual singular value decomposition presented in \cite{Brain}.
\begin{definition}
Let $\widehat A\in \mathbb{DC}^{m\times n}$. A factorization of the form
$\widehat A=\widehat U\widehat\Sigma \widehat V^{*}$
is called a compact dual singular value decomposition (CDSVD) of $\widehat A$ if there exist $\widehat U\in \mathbb{DC}^{m\times r}$ and
$\widehat V\in \mathbb{DC}^{n\times r}$ both have orthonormal columns, and
$\widehat\Sigma\in \mathbb{DR}^{r\times r}$
is diagonal positive.
\end{definition}
The equivalence between the existence of the DMPGI and that of a compact dual singular value decomposition is established
in \cite[Theorem 5.2]{Brain}.
\begin{lemma}\label{MP-CDSVD-equivalent}
Let \(\widehat A=A+\varepsilon A_0\in\mathbb{DC}^{m\times n}\) with \(m\ge n\). Assume that $A=U\Sigma V^*$
is a compact SVD of \(A\), where \(U\in\mathbb{C}^{m\times r}\) and \(V\in\mathbb{C}^{n\times r}\) both have orthonormal columns, and \(\Sigma\in\mathbb{R}^{r\times r}\) is diagonal positive. Then the following statements are equivalent:
\begin{enumerate}
\item $(I_m-UU^*)A_0(I_n-VV^*)=O_{m\times n}$;
\item \(\widehat A\) admits a CDSVD;
\item The DMPGI \(\widehat A^\dagger\) of \(\widehat A\) exists, and
\begin{equation*}
\widehat A^\dagger=
V\Sigma^{-1}U^*
+\varepsilon\Big[
(I_n-VV^*)A_0^*U\Sigma^{-2}U^*
+V\Sigma^{-2}V^*A_0^*(I_m-UU^*)
- V\Sigma^{-1}U^*A_0V\Sigma^{-1}U^*
\Big].
\end{equation*}
\end{enumerate}
\end{lemma}

To establish the corresponding result for dual complex matrices, we recall the following classical result over $\mathbb C$.
\begin{lemma}\label{completion-orthonormal-columns}\cite{GL2013}
If $U\in\mathbb{C}^{m\times n}$ with $m>n$ has orthonormal columns, then there exists
$U_{\perp}\in\mathbb{C}^{m\times(m-n)}$ such that $[\,U\ \ U_{\perp}\,]
\in\mathbb{C}^{m\times m}$
is unitary.
\end{lemma}

\subsection{Auxiliary results}
The following lemma collects basic properties of the DMPGI that can be proved using the four Penrose equations in \eqref{defdmpgi}.

\begin{lemma}\label{lem:DMPGI-properties}
Let $\widehat H\in\mathbb{DC}_z^{n\times n}$ and
$\widehat A\in\mathbb{DC}^{m\times n}$. Then the following statements hold:
\begin{itemize}
    \item[$(i)$] $\widehat H$ is Hermitian if and only if
    $\widehat H^\dagger$ is Hermitian.

    \item[$(ii)$] If $\widehat A^\dagger$ exists, then
    $(\widehat A\widehat A^*)^\dagger
    =(\widehat A^\dagger)^*\widehat A^\dagger$ and $(\widehat A^*\widehat A)^\dagger
    =\widehat A^\dagger(\widehat A^\dagger)^*.$

    \item[$(iii)$] If $\widehat A^\dagger$ exists, then
    $
    (\widehat A^\dagger)^*
    =(\widehat A^*)^\dagger.
    $
\end{itemize}
\end{lemma}

We next establish a unitary completion result
for dual complex matrices.

\begin{lemma}\label{extend unitary}
If $\widehat U\in \mathbb{DC}^{m\times n}\ (m>n)$ has orthonormal columns, then there exists
$\widehat U_{\perp}\in \mathbb{DC}^{m\times (m-n)}$ such that
$\big[\widehat U\ \ \widehat U_{\perp}\big]\in \mathbb{DC}^{m\times m}$ is dual unitary.
\end{lemma}

\begin{proof}
Since $\widehat U $ has orthonormal columns,  we have $\widehat U^{*}\widehat U=I_n$, $U^{*}U=I_n,$ and $U^{*}U_0+U_0^{*}U=O_n.$
Hence $U$ has orthonormal columns in $\mathbb C^{m\times n}$. By Lemma \ref{completion-orthonormal-columns}, there exists $U_{\perp}\in\mathbb C^{m\times(m-n)}$ such that
$Q:=[\,U\ \ U_{\perp}\,]\in\mathbb C^{m\times m}$
is unitary. In particular,
\[
U^{*}U_{\perp}=O,\qquad
U_{\perp}^{*}U_{\perp}=I_{m-n}.
\]
We choose
$U_{0,\perp}:=-\,U\big(U_0^{*}U_{\perp}\big)
\in\mathbb C^{m\times(m-n)},$
and define
$\widehat U_{\perp}:=U_{\perp}+\varepsilon U_{0,\perp}.$
We claim that
$
\widehat Q=[\,\widehat U\ \ \widehat U_{\perp}\,]
$
is dual unitary. Indeed,
\[
\widehat Q
=
Q+\varepsilon E,
\qquad
E:=[\,U_0\ \ U_{0,\perp}\,].
\]
Since $Q$ is unitary, it suffices to verify that
$Q^{*}E+E^{*}Q=O_m$.
By block multiplication,
\[
Q^{*}E+E^{*}Q
=
\begin{bmatrix}
U^{*}U_0+U_0^{*}U &
U^{*}U_{0,\perp}+U_0^{*}U_{\perp}\\[1mm]
U_{\perp}^{*}U_0+U_{0,\perp}^{*}U &
U_{\perp}^{*}U_{0,\perp}+U_{0,\perp}^{*}U_{\perp}
\end{bmatrix}.
\]
The $(1,1)$-block is zero by $U^{*}U_0+U_0^{*}U=O_n.$
Moreover, from the definition of $U_{0,\perp}$,
\[
U^{*}U_{0,\perp}
=
-\,U^{*}U\big(U_0^{*}U_{\perp}\big)
=
-\,U_0^{*}U_{\perp},
\]
and hence
$U^{*}U_{0,\perp}+U_0^{*}U_{\perp}=O.$
Taking conjugate transposes gives
$U_{\perp}^{*}U_0+U_{0,\perp}^{*}U=O.$
Furthermore, since $U^{*}U_{\perp}=O$, we have
\[
U_{\perp}^{*}U_{0,\perp}
=
-\,U_{\perp}^{*}U\big(U_0^{*}U_{\perp}\big)
=
O,
\]
and thus
$U_{\perp}^{*}U_{0,\perp}+U_{0,\perp}^{*}U_{\perp}=O.$
Therefore, $Q^{*}E+E^{*}Q=O_m.$
It follows that
\[
\widehat Q^{*}\widehat Q
=
(Q+\varepsilon E)^{*}(Q+\varepsilon E)
=
Q^{*}Q+\varepsilon\big(Q^{*}E+E^{*}Q\big)
=
I_m.
\]
Therefore, $\widehat Q$ is dual unitary. The proof is complete.
\end{proof}
The DMPGI of $\widehat X$ admits the following
representations in terms of the DMPGIs of
$\widehat X\widehat X^*$ and $\widehat X^*\widehat X$.
\begin{lemma}\label{1.4}
Let $\widehat X\in \mathbb{DC}_z^{m\times n}$. Then
$\widehat X^\dagger=\widehat X^*(\widehat X\widehat X^*)^\dagger
=(\widehat X^*\widehat X)^\dagger \widehat X^*.$
\end{lemma}

\begin{proof}
Since $\widehat X\in\mathbb{DC}_z^{m\times n}$, Lemma \ref{MP-CDSVD-equivalent} yields $\widehat X$ admits a CDSVD. Then there exist
$\widehat U_1\in\mathbb{DC}^{m\times r}$ and $\widehat V_1\in\mathbb{DC}^{n\times r},$
with orthonormal columns, and a diagonal positive dual matrix $\widehat\Sigma$ such that
\[
\widehat X=\widehat U_1\widehat\Sigma \widehat V_1^*.
\]
By Lemma \ref{extend unitary}, $\widehat U_1$ and $\widehat V_1$ can be extended to dual unitary matrices
\[
\widehat U=[\,\widehat U_1\ \ \widehat U_2\,]\in\mathbb{DC}^{m\times m},
\qquad
\widehat V=[\,\widehat V_1\ \ \widehat V_2\,]\in\mathbb{DC}^{n\times n}.
\]
Therefore,
$
\widehat X
=
\widehat U
\begin{bmatrix}
\widehat\Sigma & 0\\
0 & 0
\end{bmatrix}
\widehat V^*.$ Obviously, the DMPGI of $\widehat X$ exists and has the form $$\widehat X^\dagger
=
\widehat V
\begin{bmatrix}
\widehat\Sigma^{-1} & 0\\
0 & 0
\end{bmatrix}
\widehat U^*.
$$
Moreover,
\[
\widehat X\widehat X^*
=
\widehat U
\begin{bmatrix}
\widehat\Sigma & 0\\
0 & 0
\end{bmatrix}
\widehat V^*\widehat V
\begin{bmatrix}
\widehat\Sigma^* & 0\\
0 & 0
\end{bmatrix}
\widehat U^*
=
\widehat U
\begin{bmatrix}
\widehat\Sigma\widehat\Sigma^* & 0\\
0 & 0
\end{bmatrix}
\widehat U^*,
\]
and similarly,
\[
\widehat X^*\widehat X
=
\widehat V
\begin{bmatrix}
\widehat\Sigma^*\widehat\Sigma & 0\\
0 & 0
\end{bmatrix}
\widehat V^*.
\]
Hence $\widehat X\widehat X^*\in\mathbb{DC}_z^{m\times m}$ and
$\widehat X^*\widehat X\in\mathbb{DC}_z^{n\times n}$, with
\[
(\widehat X\widehat X^*)^\dagger
=
\widehat U
\begin{bmatrix}
(\widehat\Sigma\widehat\Sigma^*)^{-1} & 0\\
0 & 0
\end{bmatrix}
\widehat U^*
\qquad \text{and} \qquad
(\widehat X^*\widehat X)^\dagger
=
\widehat V
\begin{bmatrix}
(\widehat\Sigma^*\widehat\Sigma)^{-1} & 0\\
0 & 0
\end{bmatrix}
\widehat V^*.
\]
Furthermore,
\[
\begin{aligned}
\widehat X^*(\widehat X\widehat X^*)^\dagger
&=
\widehat V
\begin{bmatrix}
\widehat\Sigma^* & 0\\
0 & 0
\end{bmatrix}
\widehat U^*
\widehat U
\begin{bmatrix}
(\widehat\Sigma\widehat\Sigma^*)^{-1} & 0\\
0 & 0
\end{bmatrix}
\widehat U^* =
\widehat V
\begin{bmatrix}
\widehat\Sigma^*(\widehat\Sigma\widehat\Sigma^*)^{-1} & 0\\
0 & 0
\end{bmatrix}
\widehat U^*.
\end{aligned}
\]
Since $\widehat\Sigma$ is an invertible diagonal matrix, we have
$\widehat\Sigma^*(\widehat\Sigma\widehat\Sigma^*)^{-1}
=\widehat\Sigma^*(\widehat\Sigma^*)^{-1}\widehat\Sigma^{-1}=\widehat\Sigma^{-1}.
$
Therefore,
\[
\widehat X^*(\widehat X\widehat X^*)^\dagger
=
\widehat V
\begin{bmatrix}
\widehat\Sigma^{-1} & 0\\
0 & 0
\end{bmatrix}
\widehat U^*
=
\widehat X^\dagger.
\]
Analogously,
\[
\begin{aligned}
(\widehat X^*\widehat X)^\dagger \widehat X^*
&=
\widehat V
\begin{bmatrix}
(\widehat\Sigma^*\widehat\Sigma)^{-1} & 0\\
0 & 0
\end{bmatrix}
\widehat V^*
\widehat V
\begin{bmatrix}
\widehat\Sigma^* & 0\\
0 & 0
\end{bmatrix}
\widehat U^*=
\widehat V
\begin{bmatrix}
(\widehat\Sigma^*\widehat\Sigma)^{-1}\widehat\Sigma^* & 0\\
0 & 0
\end{bmatrix}
\widehat U^*=\widehat V
\begin{bmatrix}
\widehat\Sigma^{-1} & 0\\
0 & 0
\end{bmatrix}
\widehat U^*
=
\widehat X^\dagger.
\end{aligned}
\]
Consequently,
$\widehat X^\dagger=\widehat X^*(\widehat X\widehat X^*)^\dagger
=(\widehat X^*\widehat X)^\dagger \widehat X^*.$ This completes the proof.
\end{proof}

We next introduce several basic properties of Hermitian dual complex matrices.

\begin{lemma}\label{inverse and Hermitian}
Let $\widehat{A}=A+\varepsilon A_0\in\mathbb{DC}^{n\times n}$. Then the following statements hold.
\begin{enumerate}
    \item[$(i)$] $\widehat{A}$ is Hermitian if and only if both $A$ and $A_0$ are Hermitian.
    
    \item[$(ii)$] If $\widehat{A}$ is invertible, then $\widehat{A}$ is Hermitian if and only if $\widehat{A}^{-1}$ is Hermitian.
\end{enumerate}
\end{lemma}

\begin{proof}
$(i)$ By definition, \(\widehat A\) is Hermitian if and only if
$\widehat A^{*}=\widehat A.$
Since
$
\widehat A^{*}=(A+\varepsilon A_{0})^{*}=A^{*}+\varepsilon A_{0}^{*},
$
it follows that
\[
\widehat A^{*}=\widehat A
\quad\Longleftrightarrow\quad
A^{*}+\varepsilon A_{0}^{*}=A+\varepsilon A_{0}.
\]
By equality of the standard and infinitesimal parts, this is equivalent to
$
A^{*}=A
$ and $A_{0}^{*}=A_{0}$.
Hence \(\widehat A\) is Hermitian if and only if both \(A\) and \(A_{0}\) are Hermitian.

$(ii)$ We first prove the sufficiency. Since \(\widehat A^{-1}=A^{-1}-\varepsilon A^{-1}A_0A^{-1}\) is Hermitian, it follows from Lemma \ref{inverse and Hermitian} (i) that
\[
(A^{-1})^*=A^{-1}=(A^*)^{-1}.
\]
Hence \(A\) is Hermitian. Moreover, the infinitesimal part \(A^{-1}A_0A^{-1}\) is Hermitian as well. Since \(A^{-1}\) is Hermitian, it follows that \(A_0\) is Hermitian. Therefore, \(\widehat A=A+\varepsilon A_0\) is Hermitian.   

We next prove the necessity. Assume that \(\widehat A=A+\varepsilon A_0\) is Hermitian. Then, by Lemma \ref{inverse and Hermitian} (i),
\(A^*=A\) and \(A_0^*=A_0\).
Since \(\widehat A\) is invertible, and \(A\) is invertible. Thus
$(A^{-1})^*=(A^*)^{-1}=A^{-1},$
so \(A^{-1}\) is Hermitian. Moreover,
\[
(A^{-1}A_0A^{-1})^*
=(A^{-1})^*A_0^*(A^{-1})^*
=A^{-1}A_0A^{-1},
\]
which shows that \(A^{-1}A_0A^{-1}\) is Hermitian. Hence, \(\widehat A^{-1}\) is Hermitian.
\end{proof}

The following corollary gives explicit formulas for the DMPGI of dual complex matrices whose standard parts have full column or full row rank.

\begin{corollary}\label{lem:full-rank-DMPGI}
Let $\widehat A=A+\varepsilon A_0\in\mathbb{DC}^{m\times n}$.
Then the following statements hold.
\begin{itemize}
    \item[$(i)$] If the standard part $A$ is of full column rank, then $\widehat A^{*}\widehat A$ is invertible, and
    \[
    \begin{aligned}
    \widehat A^\dagger
    =(\widehat A^{*}\widehat A)^{-1}\widehat A^{*} =(A^{*}A)^{-1}A^{*}
    +\varepsilon\Big[
    (A^{*}A)^{-1}A_0^{*}
    -(A^{*}A)^{-1}
    (A^{*}A_0+A_0^{*}A)
    (A^{*}A)^{-1}A^{*}
    \Big].
    \end{aligned}
    \]

    \item[$(ii)$] If the standard part $A$ is of full row rank, then
    $\widehat A\widehat A^{*}$ is invertible, and
    \[
    \begin{aligned}
    \widehat A^\dagger
    =\widehat A^{*}
    (\widehat A\widehat A^{*})^{-1} =A^{*}(AA^{*})^{-1}
    +\varepsilon\Big[
    A_0^{*}(AA^{*})^{-1}
    -A^{*}(AA^{*})^{-1}
    (AA_0^{*}+A_0A^{*})
    (AA^{*})^{-1}
    \Big].
    \end{aligned}
    \]
\end{itemize}
\end{corollary}

\begin{proof}$(i)$
Since the standard part \(A\) is of full column rank. Thus, for any nonzero vector \(x\in\mathbb C^n\),
\[
x^*A^*Ax=(Ax)^*(Ax)=\|Ax\|^2>0.
\]
Hence \(A^*A\) is Hermitian positive definite, and therefore invertible. Then
\[
\widehat A^{\,*}\widehat A
=
(A^*+\varepsilon A_0^*)(A+\varepsilon A_0)
=
A^*A+\varepsilon(A^*A_0+A_0^*A).
\]
Since the standard part \(A^{*}A\) is invertible, the dual matrix \(\widehat A^{*}\widehat A\) is also invertible. It follows that
\[
(\widehat A^{*}\widehat A)^{-1}
=(A^{*}A)^{-1}
-\varepsilon (A^{*}A)^{-1}\bigl(A^{*}A_{0}+A_{0}^{*}A\bigr)(A^{*}A)^{-1}.
\]
Consequently, by Lemma \ref{1.4}, we have
$
\widehat A^\dagger
=(\widehat A^{*}\widehat A)^{-1}\widehat A^{*}.$

The proof of $(ii)$ follows similarly.
\end{proof}

Under the rank assumptions, the DMPGI of a product satisfies the following reverse order law.
\begin{corollary}\label{reverse order}
Let $\widehat U\in \mathbb{DC}^{m\times r}$ and $\widehat V\in \mathbb{DC}^{r\times n}$.
Assume that $U$ has full column rank and $V$ has full row rank. Then $(\widehat U\widehat V)^\dagger= \widehat V^\dagger \widehat U^\dagger.$
\end{corollary}

\begin{proof}
Let $\widehat Z=\widehat V^\dagger \widehat U^\dagger.$
Since \( U\) has full column rank and \( V\) has full row rank, by Corollary \ref{lem:full-rank-DMPGI}, we have
$\widehat U^\dagger=(\widehat U^{*}\widehat U)^{-1}\widehat U^{*},$ and $\widehat V^\dagger=\widehat V^{*}(\widehat V\widehat V^{*})^{-1},$
then
$\widehat U^\dagger \widehat U=I_r$ and $\widehat V\widehat V^\dagger=I_r.$
Thus, combining the defining equations of the DMPGI in \eqref{defdmpgi}, we conclude that $\widehat Z$ is the DMPGI of $\widehat U\widehat V$.
\end{proof}

\section{DMPGIs of sums of dual complex matrices}
We first derive a representation for the DMPGI of
$[\,\widehat U\ \widehat V\,]$ in terms of its blocks. This representation is then used to obtain formulas for the DMPGIs of sums of dual complex matrices.
\begin{lemma}\label{ck}
Let $\widehat A=[\,\widehat U\ \ \widehat V\,]\in \mathbb{DC}_z^{m\times n},$ where \(\widehat U\in \mathbb{DC}_z^{m\times k}\), \(\widehat V\in \mathbb{DC}^{m\times (n-k)}\), and \(0\le k\le n\). Define $\widehat C=(I-\widehat U\widehat U^\dagger)\widehat V.$ Assume that \(\widehat C\in \mathbb{DC}_z^{m\times (n-k)}\) and that
\begin{align}\label{K}
\widehat K=
\Big[I+(I-\widehat C^\dagger \widehat C)\widehat V^{*}
(\widehat U^\dagger)^{*}\widehat U^\dagger
\widehat V(I-\widehat C^\dagger \widehat C)\Big]^{-1}
\end{align}
exists. Then $\widehat A^\dagger$ can be written in the form
\begin{align}\label{2.1}
\widehat A^\dagger=
\begin{bmatrix}
\widehat U^\dagger
-\widehat U^\dagger \widehat V\widehat C^\dagger
-\widehat U^\dagger \widehat V
(I-\widehat C^\dagger \widehat C)\widehat K
\widehat V^{*}(\widehat U^\dagger)^{*}\widehat U^\dagger
(I-\widehat V\widehat C^\dagger)
\\[1ex]
\widehat C^\dagger
+(I-\widehat C^\dagger \widehat C)\widehat K
\widehat V^{*}(\widehat U^\dagger)^{*}\widehat U^\dagger
(I-\widehat V\widehat C^\dagger)
\end{bmatrix}.
\end{align}
\end{lemma}

\begin{proof}
Let $E_{\widehat C}:=I-\widehat C^\dagger\widehat C$, $F_{\widehat C}:=I-\widehat C\widehat C^\dagger,$
and define
$\widehat R:=E_{\widehat C}\widehat K\widehat V^*
\widehat U^{\dagger *}\widehat U^\dagger(I-\widehat V\widehat C^\dagger).$
Then $\widehat A^\dagger$ admits the following more compact expression.
\[
\widehat G:=
\begin{bmatrix}
\widehat U^\dagger-\widehat U^\dagger \widehat V\widehat C^\dagger-\widehat U^\dagger\widehat V\widehat R
\\[1mm]
\widehat C^\dagger+\widehat R
\end{bmatrix}.
\]
It suffices to verify that \(\widehat G\) satisfies the four equations in \eqref{defdmpgi} for \(\widehat A\).

We begin with several identities that will be used repeatedly. From
$\widehat C=(I-\widehat U\widehat U^\dagger)\widehat V,$
we have
\[
\widehat U^\dagger\widehat C
=
\widehat U^\dagger(I-\widehat U\widehat U^\dagger)\widehat V
=O,
\qquad
\widehat C^\dagger\widehat U=O,
\qquad
\widehat C^\dagger\widehat V
=
\widehat C^\dagger(I-\widehat U\widehat U^\dagger)\widehat V
=
\widehat C^\dagger\widehat C.
\]
Since $E_{\widehat C}\widehat C^\dagger=O$ and
$
\widehat CE_{\widehat C}
=O,
$
it follows immediately that
\[
\widehat C\widehat R
=
\widehat C E_{\widehat C}\widehat K\widehat V^*
\widehat U^{\dagger *}\widehat U^\dagger(I-\widehat V\widehat C^\dagger)
=O.
\]
In addition, since $E_{\widehat C}$ is Hermitian idempotent and
$\widehat K^{-1}
=
I+E_{\widehat C}\widehat V^*
\widehat U^{\dagger *}\widehat U^\dagger\widehat VE_{\widehat C}
$ is also Hermitian,
we have $E_{\widehat C}\widehat K^{-1}=\widehat K^{-1}E_{\widehat C}.$
Then multiplying $\widehat K$ on both sides of the above equation, we have $E_{\widehat C}\widehat K=\widehat KE_{\widehat C}$, and further obtain
\begin{align*}
E_{\widehat C}\widehat V^*
\widehat U^{\dagger *}\widehat U^\dagger\widehat VE_{\widehat C}\widehat K
=
I-\widehat K=\widehat KE_{\widehat C}\widehat V^*
\widehat U^{\dagger *}\widehat U^\dagger\widehat VE_{\widehat C}=\widehat R\widehat V.
\end{align*}
We next verify the Penrose equations in \eqref{defdmpgi} for $\widehat G$. A direct calculation gives
\[
\begin{aligned}
\widehat A\widehat G
&=
\widehat U
\bigl(\widehat U^\dagger-\widehat U^\dagger \widehat V\widehat C^\dagger-\widehat U^\dagger\widehat V\widehat R\bigr)+
\widehat V(\widehat C^\dagger+\widehat R)
=
\widehat U\widehat U^\dagger-\widehat U\widehat U^\dagger\widehat V\widehat C^\dagger
-\widehat U\widehat U^\dagger\widehat V\widehat R
+\widehat V\widehat C^\dagger+\widehat V\widehat R.
\end{aligned}
\]
Since $\widehat V=(\widehat U\widehat U^\dagger)\widehat V+\widehat C$
and $\widehat C\widehat R=O$, we have 
$
\widehat V\widehat R=\widehat U\widehat U^\dagger\widehat V\widehat R.
$
Hence
\[
\widehat A\widehat G
=
\widehat U\widehat U^\dagger
-\widehat U\widehat U^\dagger\widehat V\widehat C^\dagger
+\widehat V\widehat C^\dagger
=
\widehat U\widehat U^\dagger
+(I-\widehat U\widehat U^\dagger)\widehat V\widehat C^\dagger
=
\widehat U\widehat U^\dagger+\widehat C\widehat C^\dagger.
\]
From both $\widehat U\widehat U^\dagger$ and $\widehat C\widehat C^\dagger$ are Hermitian, it follows that
\[
(\widehat A\widehat G)^*
=
(\widehat U\widehat U^\dagger)^*+(\widehat C\widehat C^\dagger)^*
=
\widehat U\widehat U^\dagger+\widehat C\widehat C^\dagger
=
\widehat A\widehat G.
\]
Moreover,
$\widehat A\widehat G\widehat A=
(\widehat U\widehat U^\dagger+\widehat C\widehat C^\dagger)[\,\widehat U\ \ \widehat V\,].$ Recall that $\widehat U\widehat U^\dagger\widehat U=\widehat U,$ $
\widehat C^\dagger\widehat U=O,$
and
$
\widehat U\widehat U^\dagger\widehat V+\widehat C\widehat C^\dagger\widehat V
=
\widehat U\widehat U^\dagger\widehat V+\widehat C
=
\widehat V,
$
Therefore,
\[
\widehat A\widehat G\widehat A=\widehat A.
\]
It remains to show that $\widehat G\widehat A\widehat G=\widehat G$ and $(\widehat G\widehat A)^*=\widehat G\widehat A$.
A direct block multiplication gives
\[
\widehat G\widehat A
=
\begin{bmatrix}
\widehat G_1\widehat U & \widehat G_1\widehat V\\[1mm]
\widehat G_2\widehat U & \widehat G_2\widehat V
\end{bmatrix},
\]
where
$
\widehat G_1=\widehat U^\dagger-\widehat U^\dagger \widehat V\widehat C^\dagger-\widehat U^\dagger\widehat V\widehat R,$ and $
\widehat G_2=\widehat C^\dagger+\widehat R.
$
Using the relations
\[
\widehat C^\dagger\widehat U=O,
\qquad
(I-\widehat V\widehat C^\dagger)\widehat U=\widehat U,
\qquad
(I-\widehat V\widehat C^\dagger)\widehat V=\widehat VE_{\widehat C},
\]
together with the definition of DMPGI, one obtains
\begin{align}\label{GA}
\widehat G\widehat A
=
\begin{bmatrix}
\widehat U^\dagger\widehat U-\widehat U^\dagger\widehat VE_{\widehat C}\widehat K\widehat V^*\widehat U^{\dagger *}
&
\widehat U^\dagger\widehat VE_{\widehat C}\widehat K
\\[1mm]
\widehat KE_{\widehat C}\widehat V^*\widehat U^{\dagger *}
&
\widehat C^\dagger\widehat C+I-\widehat K
\end{bmatrix}.
\end{align}
Since \(\widehat K^{-1}\) is Hermitian, Lemma \ref{inverse and Hermitian} yields
\(\widehat K\) is also Hermitian.
Recall that $\widehat U^\dagger\widehat U$ and $\widehat C^\dagger\widehat C$ are Hermitian, and
$E_{\widehat C}\widehat K=\widehat KE_{\widehat C}$, the above block representation shows that
\[
(\widehat G\widehat A)^*=\widehat G\widehat A.
\]
Furthermore, multiplying \eqref{GA} by $\widehat G$ on the right. Repeated use of the identities
\[
\widehat U^\dagger\widehat C=O,
\qquad
\widehat C\widehat R=O,
\qquad E_{\widehat C}\widehat K=\widehat KE_{\widehat C},\qquad
E_{\widehat C}\widehat C^\dagger=O,
\qquad
\widehat R= E_{\widehat C}\widehat R,
\]
together with
\[
E_{\widehat C}\widehat V^*
\widehat U^{\dagger *}\widehat U^\dagger\widehat V E_{\widehat C}\widehat K=\widehat KE_{\widehat C}\widehat V^*(\widehat U^\dagger)^*\widehat U^\dagger \widehat VE_{\widehat C}=I-\widehat K,
\]
shows that the upper block of $(\widehat G\widehat A)\widehat G$ reduces to $\widehat G_1$ as follows:
\[
\begin{aligned}
&\bigl(\widehat U^\dagger \widehat U-\widehat U^\dagger \widehat VE_{\widehat C}\widehat K\widehat V^{*}(\widehat U^\dagger)^{*}\bigr)\widehat G_1
+\widehat U^\dagger\widehat VE_{\widehat C}\widehat K\widehat G_2\\
={}&\bigl(\widehat U^\dagger \widehat U-\widehat U^\dagger \widehat VE_{\widehat C}\widehat K\widehat V^{*}(\widehat U^\dagger)^{*}\bigr)
\bigl(\widehat U^\dagger-\widehat U^\dagger\widehat V\widehat C^\dagger-\widehat U^\dagger\widehat V\widehat R\bigr)
+\widehat U^\dagger\widehat V E_{\widehat C}\widehat K(\widehat C^\dagger+\widehat R)\\
={}&\widehat U^\dagger
-\widehat U^\dagger\widehat V\widehat C^\dagger
-\widehat U^\dagger\widehat V\widehat R
-\widehat U^\dagger\widehat VE_{\widehat C}\widehat K\widehat V^*(\widehat U^\dagger)^*\widehat U^\dagger(I-\widehat V\widehat C^\dagger)
+\widehat U^\dagger\widehat VE_{\widehat C}\widehat K\widehat V^*(\widehat U^\dagger)^*\widehat U^\dagger\widehat V\widehat R\\
&
+\widehat U^\dagger\widehat V E_{\widehat C}\widehat K\widehat C^\dagger
+\widehat U^\dagger\widehat V E_{\widehat C}\widehat K\widehat R\\
={}&\widehat U^\dagger
-\widehat U^\dagger\widehat V\widehat C^\dagger
-\widehat U^\dagger\widehat V\widehat R
-\widehat U^\dagger\widehat V\widehat R
+\widehat U^\dagger\widehat VE_{\widehat C}\widehat K\widehat V^*(\widehat U^\dagger)^*\widehat U^\dagger\widehat V\widehat R
+\widehat U^\dagger\widehat V E_{\widehat C}\widehat K\widehat R\\
={}&\widehat U^\dagger
-\widehat U^\dagger\widehat V\widehat C^\dagger
-\widehat U^\dagger\widehat V\widehat R
-\widehat U^\dagger\widehat V\widehat R
+\widehat U^\dagger\widehat V(I-\widehat K)\widehat R
+\widehat U^\dagger\widehat V\widehat K\widehat R\\
={}&\widehat U^\dagger
-\widehat U^\dagger\widehat V\widehat C^\dagger
-\widehat U^\dagger\widehat V\widehat R\\
={}&\widehat G_1.
\end{aligned}
\]
Recall that \[
\widehat C\widehat R=O,
\qquad
\widehat C^\dagger\widehat C\widehat C^\dagger=\widehat C^\dagger,
\qquad
(I-\widehat K)\widehat C^\dagger=O,
\qquad
\widehat R=E_{\widehat C}\widehat R,\]
the lower block of $(\widehat G\widehat A)\widehat G$ reduces to $\widehat G_2$ as follows:
\[
\begin{aligned}
&\widehat KE_{\widehat C}\widehat V^*(\widehat U^\dagger)^*\widehat G_1
+\bigl(\widehat C^\dagger\widehat C+I-\widehat K\bigr)\widehat G_2\\
={}&\widehat KE_{\widehat C}\widehat V^*(\widehat U^\dagger)^*
\bigl(\widehat U^\dagger-\widehat U^\dagger\widehat V\widehat C^\dagger-\widehat U^\dagger\widehat V\widehat R\bigr)
+\bigl(\widehat C^\dagger\widehat C+I-\widehat K\bigr)(\widehat C^\dagger+\widehat R)\\
={}&\widehat KE_{\widehat C}\widehat V^*(\widehat U^\dagger)^*\widehat U^\dagger(I-\widehat V\widehat C^\dagger)
-\widehat KE_{\widehat C}\widehat V^*(\widehat U^\dagger)^*\widehat U^\dagger\widehat V\widehat R
+\widehat C^\dagger\widehat C\widehat C^\dagger
\\&+\widehat C^\dagger\widehat C\widehat R+(I-\widehat K)\widehat C^\dagger+(I-\widehat K)\widehat R\\
={}&\widehat R
-\widehat KE_{\widehat C}\widehat V^*(\widehat U^\dagger)^*\widehat U^\dagger\widehat VE_{\widehat C}\widehat R
+\widehat C^\dagger+O+O+(I-\widehat K)\widehat R\\
={}&\widehat R-(I-\widehat K)\widehat R+\widehat C^\dagger+(I-\widehat K)\widehat R\\
={}&\widehat C^\dagger+\widehat R\\
={}&\widehat G_2.
\end{aligned}
\]
Consequently,
$
\widehat G\widehat A\widehat G=\widehat G$. Thus $\widehat G$ satisfies all four Penrose equations for $\widehat A$. By the uniqueness of the DMPGI, we conclude that
$\widehat G=\widehat A^\dagger.
$ This completes the proof.
\end{proof}

Combining the preceding block representation with the DMPGI properties yields the following formula.
\begin{theorem}\label{uuvv}
    Let $\widehat U,\widehat V$ be any dual complex matrices and let $\widehat C$ and $\widehat K$ be defined as in \eqref{K}. Assume that all dual matrices involved in the following expression possess DMPGIs. Then the DMPGI of $\widehat U\widehat U^*+\widehat V\widehat V^*$ can be written in the form
    \begin{align*}
        (\widehat U\widehat U^*+\widehat V\widehat V^*)^\dagger=
\big(I-(\widehat C^\dagger)^*\widehat V^{*}\big)\,(\widehat U^\dagger)^*
\Bigl[I-\widehat U^{\dagger}\widehat V
(I-\widehat C^{\dagger}\widehat C)\widehat K\widehat V^{*}(\widehat U^\dagger)^*
\Bigr]\widehat U^{\dagger}
(I-\widehat V\widehat C^{\dagger})
+(\widehat C^\dagger)^*
\widehat C^{\dagger}.
\end{align*}
\end{theorem}
\begin{proof}
Let $\widehat A=[\,\widehat U \ \ \widehat V\,].$
By Lemma \ref{lem:DMPGI-properties}, we have
\begin{equation}\label{dual-2.2}
(\widehat U\widehat U^{*}+\widehat V\widehat V^{*})^{\dagger}
=
(\widehat A\widehat A^{*})^{\dagger}
=
(\widehat A^*)^\dagger\widehat A^{\dagger}=(\widehat A^\dagger)^*\widehat A^{\dagger}.
\end{equation}
The representation of $(\widehat U\widehat U^{*}+\widehat V\widehat V^{*})^{\dagger}$ is obtained by substituting the expression of $\widehat A^{\dagger}$ from \eqref{2.1} into the right-hand side of \eqref{dual-2.2} and then carrying out block multiplication. For this purpose, define
\begin{equation}\label{eq:dual-2.3M}
\widehat M
=
I-\widehat U^{\dagger}\widehat V
\bigl(I-\widehat C^{\dagger}\widehat C\bigr)
\widehat K\,\widehat V^{*}(\widehat U^\dagger)^*,
\end{equation}
and
\begin{equation*}\label{eq:dual-2.3N}
\widehat N
=
\bigl(I-\widehat C^{\dagger}\widehat C\bigr)
\widehat K\,\widehat V^{*}(\widehat U^\dagger)^*\widehat U^{\dagger}
\bigl(I-\widehat V\widehat C^{\dagger}\bigr),
\end{equation*}
so that \eqref{2.1} can be rewritten as
\begin{equation}\label{eq:dual-2.4}
\widehat A^{\dagger}
=
\begin{bmatrix}
\widehat M\widehat U^{\dagger}(I-\widehat V\widehat C^{\dagger})\\[1mm]
\widehat C^{\dagger}+\widehat N
\end{bmatrix}.
\end{equation}
By the defining equations of the DMPGI, $I-\widehat C^{\dagger}\widehat C$ is Hermitian and $\bigl(I-\widehat C^{\dagger}\widehat C\bigr)\widehat C^{\dagger}=O.$ Consequently,
$\widehat N^{*}\widehat C^{\dagger}=O,$
and therefore
\begin{equation}\label{eq:dual-2.5}
(\widehat C^{\dagger}+\widehat N)^{*}(\widehat C^{\dagger}+\widehat N)
=
(\widehat C^\dagger)^*\widehat C^{\dagger}
+
\widehat N^{*}\widehat N.
\end{equation}
Next, since $I-\widehat C^{\dagger}\widehat C$ is idempotent and commutes with the dual Hermitian matrix $\widehat K$, we have
\[
\begin{aligned}
\widehat M^{*}\widehat M
&=I-2\widehat U^{\dagger}\widehat V
\bigl(I-\widehat C^{\dagger}\widehat C\bigr)
\widehat K\widehat V^{*}(\widehat U^\dagger)^*
\\&+\widehat U^{\dagger}\widehat V
\bigl(I-\widehat C^{\dagger}\widehat C\bigr)
\widehat K
\bigl(I-\widehat C^{\dagger}\widehat C\bigr)
\widehat V^{*}(\widehat U^\dagger)^*
\widehat U^{\dagger}\widehat V
\bigl(I-\widehat C^{\dagger}\widehat C\bigr)
\widehat K\widehat V^{*}(\widehat U^\dagger)^*.
\end{aligned}
\]
Using the defining relation of $\widehat K$,
\[
\bigl(I-\widehat C^{\dagger}\widehat C\bigr)
\widehat V^{*}(\widehat U^\dagger)^*\widehat U^{\dagger}\widehat V
\bigl(I-\widehat C^{\dagger}\widehat C\bigr)\widehat K
=
I-\widehat K,
\]
it follows that
\begin{equation}\label{eq:dual-2.6}
\widehat M^{*}\widehat M
=I-\widehat U^{\dagger}\widehat V
\bigl(I-\widehat C^{\dagger}\widehat C\bigr)
\widehat K\widehat V^{*}(\widehat U^\dagger)^*
-
\widehat U^{\dagger}\widehat V
\bigl(I-\widehat C^{\dagger}\widehat C\bigr)
\widehat K^{2}\widehat V^{*}(\widehat U^\dagger)^*.
\end{equation}

Observe that multiplying the last term in \eqref{eq:dual-2.6} on the left by
$\bigl(I-(\widehat C^\dagger)^*\widehat V^{*}\bigr)(\widehat U^\dagger)^*$
and on the right by $\widehat U^{\dagger}(I-\widehat V\widehat C^{\dagger})$
yields $-\widehat N^{*}\widehat N$. Hence, combining \eqref{eq:dual-2.4}, \eqref{eq:dual-2.5}, and \eqref{eq:dual-2.6}, we obtain
\[
(\widehat A^\dagger)^*\widehat A^{\dagger}
=
\bigl(I-(\widehat C^\dagger)^*\widehat V^{*}\bigr)
(\widehat U^\dagger)^*\widehat M\widehat U^{\dagger}
\bigl(I-\widehat V\widehat C^{\dagger}\bigr)
+
(\widehat C^\dagger)^*\widehat C^{\dagger}.
\]
Substituting the expression of $\widehat M$ from \eqref{eq:dual-2.3M} into the above identity, and using \eqref{dual-2.2}, we obtain
\[
(\widehat U\widehat U^{*}+\widehat V\widehat V^{*})^{\dagger}
=
\bigl(I-(\widehat C^\dagger)^*\widehat V^{*}\bigr)(\widehat U^\dagger)^*
\Bigl[
I-\widehat U^{\dagger}\widehat V
\bigl(I-\widehat C^{\dagger}\widehat C\bigr)\widehat K
\widehat V^{*}(\widehat U^\dagger)^*
\Bigr]
\widehat U^{\dagger}\bigl(I-\widehat V\widehat C^{\dagger}\bigr)
+(\widehat C^\dagger)^*\widehat C^{\dagger}.
\]
This completes the proof.
\end{proof}

Under the condition $\widehat U\widehat V^*=O$,
the preceding result yields the following representation for $(\widehat U+\widehat V)^\dagger$.
\begin{theorem}\label{U+VUV*}
Let \(\widehat U,\widehat V\in\mathbb{DC}^{m\times n}\) satisfy
$\widehat U\widehat V^{\,*}=O.$
Assume that \(\widehat U\) and
$\widehat\Delta=\bigl(I-\widehat U\widehat U^\dagger\bigr)\widehat V
$ admit the DMPGI. Then
\[
(\widehat U+\widehat V)^\dagger
=
\widehat U^\dagger
+
\bigl(I-\widehat U^\dagger \widehat V\bigr)
\Big[
\widehat\Delta^\dagger
+
\bigl(I-\widehat\Delta^\dagger\widehat\Delta\bigr)\widehat\Phi\,
\widehat V^{\,*}(\widehat U^\dagger)^*
\widehat U^\dagger
\bigl(I-\widehat V\widehat\Delta^\dagger\bigr)
\Big],
\]
where
$
\widehat\Phi=
\Big[I+
\bigl(I-\widehat\Delta^\dagger\widehat\Delta\bigr)
\widehat V^{\,*}
(\widehat U^\dagger)^{*}
\widehat U^\dagger
\widehat V
\bigl(I-\widehat\Delta^\dagger\widehat\Delta\bigr)
\Big]^{-1}.
$
\end{theorem}

\begin{proof}
Assume that $\widehat A=[\widehat U \ \ \widehat V]$ and $\widehat U\widehat V^*=O$, then $$\widehat A\widehat A^*=\widehat U\widehat U^*+\widehat V\widehat V^*=(\widehat U+\widehat V)(\widehat U+\widehat V)^*,$$
and it follows from Lemma \ref{1.4} that
\begin{align}\label{n1n2}
    (\widehat U+\widehat V)^\dagger&=(\widehat U+\widehat V)^*(\widehat U\widehat U^*+\widehat V\widehat V^*)^\dagger=\widehat U^*(\widehat U\widehat U^*+\widehat V\widehat V^*)^\dagger+\widehat V^*(\widehat U\widehat U^*+\widehat V\widehat V^*)^\dagger.
\end{align}
From Theorem \ref{uuvv}, we have a general form for $(\widehat U\widehat U^*+\widehat V\widehat V^*)^\dagger$ which can be substituted directly into \eqref{n1n2}. Alternatively, note that applying Lemma \ref{1.4} to the partitioned dual matrix $\widehat A=[\widehat U \ \ \widehat V]$ gives
\begin{align}\label{3.2}
    \widehat A^\dagger=
\widehat A^{*}(\widehat A\widehat A^{*})^{\dagger}
=
\begin{bmatrix}
\widehat U^{*}(\widehat U\widehat U^{*}+\widehat V\widehat V^{*})^{\dagger}\\[1mm]
\widehat V^{*}(\widehat U\widehat U^{*}+\widehat V\widehat V^{*})^{\dagger}
\end{bmatrix}.
\end{align}
Since $\widehat A^\dagger$ is unique, submatrices in \eqref{2.1} and \eqref{3.2} must be equal. The proof is complete.
\end{proof}

\section{Numerical examples}
This section presents two numerical examples to illustrate the preceding theorems. The first example illustrates Theorem~\ref{uuvv}.
\begin{example}\label{ex:uuvv}
Let $\varepsilon$ be the dual unit satisfying
$\varepsilon^2=0$, and consider the dual matrices
\[
\widehat U=
\begin{bmatrix}
1+\varepsilon&0\\
0&0\\
0&0
\end{bmatrix},
\qquad
\widehat V=
\begin{bmatrix}
2+4\varepsilon&0\\
1+3\varepsilon&1+3\varepsilon\\
\varepsilon&\varepsilon
\end{bmatrix}.
\]
For the matrices above, all DMPGIs and inverses required
in Theorem~\ref{uuvv} exist.
We then apply Theorem \ref{uuvv} to compute
$(\widehat U\widehat U^*+\widehat V\widehat V^*)^\dagger$,
beginning with $\widehat U^\dagger$ and
$\widehat C^\dagger$.

Since $(1+\varepsilon)^{-1}=1-\varepsilon$, we have
\[
\widehat U^\dagger=
\begin{bmatrix}
1-\varepsilon&0&0\\
0&0&0
\end{bmatrix},
\qquad
I-\widehat U\widehat U^\dagger=
\begin{bmatrix}
0&0&0\\
0&1&0\\
0&0&1
\end{bmatrix}.
\]
Consequently,
\[
\widehat C
=(I-\widehat U\widehat U^\dagger)\widehat V
=
\begin{bmatrix}
0&0\\
1+3\varepsilon&1+3\varepsilon\\
\varepsilon&\varepsilon
\end{bmatrix}.
\]
The DMPGI of $\widehat C$ is given by
\begin{equation}\label{C+}
\widehat C^\dagger=
\begin{bmatrix}
0&\frac12-\frac32\varepsilon&\frac12\varepsilon\\[1mm]
0&\frac12-\frac32\varepsilon&\frac12\varepsilon
\end{bmatrix}.
\end{equation}
Indeed,
\[
\widehat C\widehat C^\dagger=
\begin{bmatrix}
0&0&0\\
0&1&\varepsilon\\
0&\varepsilon&0
\end{bmatrix}=(\widehat C\widehat C^\dagger)^*,
\qquad
\widehat C^\dagger\widehat C
=\frac12
\begin{bmatrix}
1&1\\
1&1
\end{bmatrix}=(\widehat C^\dagger\widehat C)^*,
\]
and $\widehat C\widehat C^\dagger\widehat C=\widehat C,~\widehat C^\dagger\widehat C\widehat C^\dagger=\widehat C^\dagger.$
Thus, the matrix in \eqref{C+} satisfies all four Penrose equations in \eqref{defdmpgi} and is therefore the DMPGI of $\widehat C$. In particular,
\[
I-\widehat C^\dagger\widehat C
=\frac12
\begin{bmatrix}
1&-1\\
-1&1
\end{bmatrix}.
\]
To compute $\widehat K$, we first obtain
\[
\widehat U^\dagger\widehat V
=
\begin{bmatrix}
2+2\varepsilon&0\\
0&0
\end{bmatrix},\qquad (\widehat U^\dagger\widehat V)^*
=
\begin{bmatrix}
2+2\varepsilon&0\\
0&0
\end{bmatrix}
\]
and hence
\[
\widehat V^*(\widehat U^\dagger)^*
\widehat U^\dagger\widehat V
=
\begin{bmatrix}
4+8\varepsilon&0\\
0&0
\end{bmatrix}.
\]
Therefore,
\[
\begin{aligned}
(I-\widehat C^\dagger\widehat C)
\widehat V^*(\widehat U^\dagger)^*
\widehat U^\dagger\widehat V
(I-\widehat C^\dagger\widehat C)
=
\frac14
\begin{bmatrix}
1&-1\\
-1&1
\end{bmatrix}
\begin{bmatrix}
4+8\varepsilon&0\\
0&0
\end{bmatrix}
\begin{bmatrix}
1&-1\\
-1&1
\end{bmatrix}
=
(1+2\varepsilon)
\begin{bmatrix}
1&-1\\
-1&1
\end{bmatrix}.
\end{aligned}
\]
Since $\widehat K=
\Big[I+(I-\widehat C^\dagger \widehat C)\widehat V^{*}
(\widehat U^\dagger)^{*}\widehat U^\dagger
\widehat V(I-\widehat C^\dagger \widehat C)\Big]^{-1}$, it follows that
\[
\begin{aligned}
\widehat K
=
\begin{bmatrix}
2+2\varepsilon&-1-2\varepsilon\\
-1-2\varepsilon&2+2\varepsilon
\end{bmatrix}^{-1}
=
\frac{1}{3+4\varepsilon}
\begin{bmatrix}
2+2\varepsilon&1+2\varepsilon\\
1+2\varepsilon&2+2\varepsilon
\end{bmatrix}
=
\begin{bmatrix}
\frac23-\frac29\varepsilon&
\frac13+\frac29\varepsilon\\[1mm]
\frac13+\frac29\varepsilon&
\frac23-\frac29\varepsilon
\end{bmatrix},
\end{aligned}
\]
Moreover,
\[
(I-\widehat C^\dagger\widehat C)\widehat K
=
\frac{1}{3+4\varepsilon}
(I-\widehat C^\dagger\widehat C),
\]
and
\[
\begin{aligned}
\widehat U^\dagger\widehat V
(I-\widehat C^\dagger\widehat C)
\widehat V^*(\widehat U^\dagger)^*
=
\frac12
\begin{bmatrix}
2+2\varepsilon&0\\
0&0
\end{bmatrix}
\begin{bmatrix}
1&-1\\
-1&1
\end{bmatrix}
\begin{bmatrix}
2+2\varepsilon&0\\
0&0
\end{bmatrix}
=
\begin{bmatrix}
2+4\varepsilon&0\\
0&0
\end{bmatrix}.
\end{aligned}
\]
Consequently,
\[
\begin{aligned}
&\widehat U^\dagger\widehat V
(I-\widehat C^\dagger\widehat C)
\widehat K\widehat V^*(\widehat U^\dagger)^*
=
\frac{1}{3+4\varepsilon}
\begin{bmatrix}
2+4\varepsilon&0\\
0&0
\end{bmatrix}
=
\begin{bmatrix}
\frac23+\frac49\varepsilon&0\\
0&0
\end{bmatrix}.
\end{aligned}
\]
The matrix inside the square brackets in
Theorem~\ref{uuvv} is therefore
\[
I-\widehat U^\dagger\widehat V
(I-\widehat C^\dagger\widehat C)
\widehat K\widehat V^*(\widehat U^\dagger)^*
=
\begin{bmatrix}
\frac13-\frac49\varepsilon&0\\
0&1
\end{bmatrix}.
\]
Then,
\[
\widehat V\widehat C^\dagger=
\begin{bmatrix}
0&1-\varepsilon&\varepsilon\\
0&1&\varepsilon\\
0&\varepsilon&0
\end{bmatrix},\qquad I-\widehat V\widehat C^\dagger
=
\begin{bmatrix}
1&-1+\varepsilon&-\varepsilon\\
0&0&-\varepsilon\\
0&-\varepsilon&1
\end{bmatrix}.
\]
It follows that
\[
\widehat U^\dagger
(I-\widehat V\widehat C^\dagger)
=
\begin{bmatrix}
1-\varepsilon&-1+2\varepsilon&-\varepsilon\\
0&0&0
\end{bmatrix},
\]
by taking conjugate transposes,
\[
\bigl(I-(\widehat C^\dagger)^*\widehat V^*\bigr)
(\widehat U^\dagger)^*
=
\begin{bmatrix}
1-\varepsilon&0\\
-1+2\varepsilon&0\\
-\varepsilon&0
\end{bmatrix}.
\]
Thus, the first summand on the right-hand side of
Theorem~\ref{uuvv} is
\[
\begin{aligned}
&\bigl(I-(\widehat C^\dagger)^*\widehat V^*\bigr)
(\widehat U^\dagger)^*
\times
\Bigl[
I-\widehat U^\dagger\widehat V
(I-\widehat C^\dagger\widehat C)
\widehat K\widehat V^*(\widehat U^\dagger)^*
\Bigr]
\widehat U^\dagger(I-\widehat V\widehat C^\dagger)
\\
&=
\left(\frac13-\frac49\varepsilon\right)
\begin{bmatrix}
1-\varepsilon\\
-1+2\varepsilon\\
-\varepsilon
\end{bmatrix}
\begin{bmatrix}
1-\varepsilon&-1+2\varepsilon&-\varepsilon
\end{bmatrix}
\\
&=
\begin{bmatrix}
\frac13&-\frac13&0\\
-\frac13&\frac13&0\\
0&0&0
\end{bmatrix}
+\varepsilon
\begin{bmatrix}
-\frac{10}{9}&\frac{13}{9}&-\frac13\\
\frac{13}{9}&-\frac{16}{9}&\frac13\\
-\frac13&\frac13&0
\end{bmatrix}.
\end{aligned}
\]
The second summand is
\[
\begin{aligned}
(\widehat C^\dagger)^*\widehat C^\dagger
&=
2
\begin{bmatrix}
0\\
\frac12-\frac32\varepsilon\\
\frac12\varepsilon
\end{bmatrix}
\begin{bmatrix}
0&\frac12-\frac32\varepsilon&\frac12\varepsilon
\end{bmatrix}
=
\begin{bmatrix}
0&0&0\\
0&\frac12&0\\
0&0&0
\end{bmatrix}
+\varepsilon
\begin{bmatrix}
0&0&0\\
0&-3&\frac12\\
0&\frac12&0
\end{bmatrix}.
\end{aligned}
\]
Combining the above expressions gives
\begin{equation*}
(\widehat U\widehat U^*
+\widehat V\widehat V^*)^\dagger
=
\begin{bmatrix}
\frac13&-\frac13&0\\
-\frac13&\frac56&0\\
0&0&0
\end{bmatrix}
+\varepsilon
\begin{bmatrix}
-\frac{10}{9}&\frac{13}{9}&-\frac13\\
\frac{13}{9}&-\frac{43}{9}&\frac56\\
-\frac13&\frac56&0
\end{bmatrix}.
\end{equation*}
\end{example}

We next illustrate Theorem~\ref{U+VUV*} by an example in which the orthogonality condition is satisfied and the correction matrix
$\widehat\Phi$ is nontrivial.
\begin{example}\label{ex:U+VUV*}
Let $\varepsilon$ be the dual unit satisfying
$\varepsilon^2=0$, and consider the dual matrices
\[
\widehat U
=(1+\varepsilon)
\begin{bmatrix}
1&0&0\\
1&0&0\\
0&0&0
\end{bmatrix},
\qquad
\widehat V
=(1+\varepsilon)
\begin{bmatrix}
0&-1&-1\\
0&-1&-1\\
0&-1&0
\end{bmatrix}.
\]
A direct calculation gives $\widehat U\widehat V^*=O.$ Moreover,
\[
\widehat U^\dagger
=(1-\varepsilon)
\begin{bmatrix}
\frac12&\frac12&0\\
0&0&0\\
0&0&0
\end{bmatrix},
\qquad
I-\widehat U\widehat U^\dagger
=
\begin{bmatrix}
\frac12&-\frac12&0\\
-\frac12&\frac12&0\\
0&0&1
\end{bmatrix}.
\]
It follows that
\[
\widehat\Delta
=(I-\widehat U\widehat U^\dagger)\widehat V
=(1+\varepsilon)
\begin{bmatrix}
0&0&0\\
0&0&0\\
0&-1&0
\end{bmatrix}.
\]
Therefore,
\[
\widehat\Delta^\dagger
=(1-\varepsilon)
\begin{bmatrix}
0&0&0\\
0&0&-1\\
0&0&0
\end{bmatrix},
\qquad
I-\widehat\Delta^\dagger\widehat\Delta
=
\begin{bmatrix}
1&0&0\\
0&0&0\\
0&0&1
\end{bmatrix}.
\]
Using the definition of $\widehat\Phi$ in
Theorem~\ref{U+VUV*}, we obtain
\[
\begin{aligned}
\widehat\Phi
=
\Big[
I+
(I-\widehat\Delta^\dagger\widehat\Delta)
\widehat V^*
(\widehat U^\dagger)^*
\widehat U^\dagger
\widehat V
(I-\widehat\Delta^\dagger\widehat\Delta)
\Big]^{-1}=
\begin{bmatrix}
1&0&0\\
0&1&0\\
0&0&\frac12
\end{bmatrix}.
\end{aligned}
\]
Substituting the above expressions into
Theorem~\ref{U+VUV*} yields
\[
(\widehat U+\widehat V)^\dagger
=
(1-\varepsilon)
\begin{bmatrix}
\frac14&\frac14&-\frac12\\
0&0&-1\\
-\frac14&-\frac14&\frac12
\end{bmatrix}.
\]
\end{example}

\section{Rigid-body motion estimation}

Rigid-body motion estimation arises in point-based registration, motion tracking, and computer-assisted surgery. The effect of
point configuration on registration accuracy has been studied in \cite{Fitzpatrick1998,West2001}. Here we consider the recovery of angular and translational velocities from the velocities of
two observation points. Placing these points symmetrically about the reference point yields a rank-deficient system in which rotation about the line joining them is unobservable. This
configuration provides an explicit example of the application of Theorem~\ref{uuvv} to motion recovery.

Let $\boldsymbol{\omega}\in\mathbb{R}^3$ denote the angular velocity of the rigid body, and let $\boldsymbol v\in\mathbb{R}^3$ denote the
velocity of a reference point fixed in the body. If $\boldsymbol r_j\in\mathbb{R}^3$ is the position of the $j$-th observation point relative to this reference point, then its
velocity is
\[
\boldsymbol b_j
=
\boldsymbol v+\boldsymbol\omega\times\boldsymbol r_j
=
\boldsymbol v-[\boldsymbol r_j]_\times\boldsymbol\omega,
\]
where
\[
[\boldsymbol r_j]_\times
=
\begin{bmatrix}
0&-r_{j,3}&r_{j,2}\\
r_{j,3}&0&-r_{j,1}\\
-r_{j,2}&r_{j,1}&0
\end{bmatrix}.
\]
Here $\boldsymbol v$ is the velocity of the reference point, while $\boldsymbol\omega\times\boldsymbol r_j$ is the velocity induced by the rotation of the rigid body about that point.

For a normalized collinear configuration, take the reference point as
the origin and let the two observation points be located symmetrically
on the $x$-axis,
\[
\boldsymbol r_1=
\begin{bmatrix}
1&
0&
0
\end{bmatrix}^T,
\qquad
\boldsymbol r_2=
\begin{bmatrix}
-1&
0&
0
\end{bmatrix}^T.
\]
Stacking the two velocity equations gives
\[
\begin{bmatrix}
\boldsymbol b_1\\
\boldsymbol b_2
\end{bmatrix}
=
\begin{bmatrix}
-[\boldsymbol r_1]_\times & I_3\\
-[\boldsymbol r_2]_\times & I_3
\end{bmatrix}
\begin{bmatrix}
\boldsymbol\omega\\
\boldsymbol v
\end{bmatrix}.
\]
Thus,
\[
\boldsymbol b=A\boldsymbol\xi,
\qquad
\boldsymbol b=
\begin{bmatrix}
\boldsymbol b_1\\
\boldsymbol b_2
\end{bmatrix},
\qquad
\boldsymbol\xi=
\begin{bmatrix}
\boldsymbol\omega\\
\boldsymbol v
\end{bmatrix},
\qquad
A=[\,U\ V\,],
\]
where
\[
U=
\begin{bmatrix}
-[\boldsymbol r_1]_\times\\
-[\boldsymbol r_2]_\times
\end{bmatrix}=\begin{bmatrix}
0&0&0\\
0&0&1\\
0&-1&0\\
0&0&0\\
0&0&-1\\
0&1&0
\end{bmatrix},
\qquad
V=
\begin{bmatrix}
I_3\\
I_3
\end{bmatrix}=\begin{bmatrix}
1&0&0\\
0&1&0\\
0&0&1\\
1&0&0\\
0&1&0\\
0&0&1
\end{bmatrix}.
\]
Here $U$ and $V$ are the rotational and translational block columns, respectively. Hence
\[
A=
\begin{bmatrix}
0&0&0&1&0&0\\
0&0&1&0&1&0\\
0&-1&0&0&0&1\\
0&0&0&1&0&0\\
0&0&-1&0&1&0\\
0&1&0&0&0&1
\end{bmatrix}.
\]
Let
\[
\boldsymbol\omega=(\omega_x,\omega_y,\omega_z)^T,
\qquad
\boldsymbol v=(v_x,v_y,v_z)^T.
\]
The velocities of the two observation points are
\[
\boldsymbol b_1=
\begin{bmatrix}
v_x\\
v_y+\omega_z\\
v_z-\omega_y
\end{bmatrix},
\qquad
\boldsymbol b_2=
\begin{bmatrix}
v_x\\
v_y-\omega_z\\
v_z+\omega_y
\end{bmatrix}.
\]
A direct calculation gives
\[
\operatorname{rank}(A)=5,
\qquad
\mathcal N(A)
=
\operatorname{span}
\left\{
\begin{bmatrix}
1&0&0&0&0&0
\end{bmatrix}^T
\right\}.
\]
This null space corresponds to rotation about the $x$-axis.
Since $\boldsymbol r_1$ and $\boldsymbol r_2$ lie on this axis,
\[
(\boldsymbol\omega+\alpha e_1)\times\boldsymbol r_j
=
\boldsymbol\omega\times\boldsymbol r_j,
\qquad j=1,2,\quad \alpha\in\mathbb R.
\]
Thus, adding an arbitrary $w_x$ keeps the velocity induced by the rotation of the rigid body about that point fixed and leaves the observed point velocities unchanged. This component is therefore unobservable, and the motion vector cannot be uniquely determined from the velocity data.

We next incorporate a first-order perturbation in the observation geometry. Let
\[
\widehat{\boldsymbol r}_1
=(1+\varepsilon)\boldsymbol r_1,
\qquad
\widehat{\boldsymbol r}_2
=(1+\varepsilon)\boldsymbol r_2,
\qquad
\varepsilon^2=0.
\]
This perturbation represents a first-order change in the positions of the two observation points along the $x$-axis while preserving the
rank-deficient collinear geometry. The corresponding dual coefficient matrix is 
\[
\widehat A
=
[\,\widehat U\ \widehat V\,],
\qquad
\widehat U=(1+\varepsilon)U,
\qquad
\widehat V=V.
\]
Since the standard part of $\widehat A$ is the rank deficient matrix
$A$, Corollary~\ref{lem:full-rank-DMPGI} is not satisfied. Nevertheless,
Lemma~\ref{1.4} gives
\begin{align}\label{AAA}
\widehat A^\dagger
=
\widehat A^*(\widehat A\widehat A^*)^\dagger.
\end{align}
Thus, it suffices to determine the DMPGI of the associated matrix
$\widehat A\widehat A^*$. From the column partition of $\widehat A$, we have
\[
\widehat A\widehat A^*
=
\widehat U\widehat U^*
+
\widehat V\widehat V^*,
\]
which is precisely the form covered by Theorem~\ref{uuvv}. 

Since $\widehat U^*\widehat V=0$, Lemma~\ref{1.4} implies
$\widehat U^\dagger\widehat V=0$. Consequently,
\[
\widehat C
=
(I-\widehat U\widehat U^\dagger)\widehat V
=
\widehat V=V,
\qquad
\widehat C^\dagger\widehat C=I_3,
\qquad
I_3-\widehat C^\dagger\widehat C=O,
\qquad
\widehat K=I_3.
\]
All DMPGIs required below exist for this configuration.
The representation in Theorem~\ref{uuvv} therefore reduces to
\begin{align}\label{aa*+}
(\widehat A\widehat A^*)^\dagger
=
(\widehat U^\dagger)^*\widehat U^\dagger
+
(\widehat V^\dagger)^*\widehat V^\dagger.
\end{align}
Writing
$\widehat A\widehat A^*=G+\varepsilon G_0,$
we have
\[
G
=
\begin{bmatrix}
1&0&0&1&0&0\\
0&2&0&0&0&0\\
0&0&2&0&0&0\\
1&0&0&1&0&0\\
0&0&0&0&2&0\\
0&0&0&0&0&2
\end{bmatrix},
\qquad
G_0=
\begin{bmatrix}
0&0&0&0&0&0\\
0&2&0&0&-2&0\\
0&0&2&0&0&-2\\
0&0&0&0&0&0\\
0&-2&0&0&2&0\\
0&0&-2&0&0&2
\end{bmatrix}.
\]
Here $\operatorname{rank}(G)=5$. Since
$\widehat U=(1+\varepsilon)U$, $\widehat V=V$, and
$\varepsilon^2=0$, we obtain
\[
\widehat U^\dagger=(1-\varepsilon)U^\dagger=(1-\varepsilon)(U^{*}U)^{\dagger}U^{*}=\frac{1-\varepsilon}{2}
\begin{bmatrix}
0 & 0 & 0  & 0 & 0  & 0 \\
0 & 0 & -1 & 0 & 0  & 1 \\
0 & 1 & 0  & 0 & -1 & 0
\end{bmatrix},
\]
and
\[
\widehat V^\dagger=V^\dagger=(V^{*}V)^{-1}V^{*}=\frac{1}{2}
\begin{bmatrix}
1 & 0 & 0 & 1 & 0 & 0 \\
0 & 1 & 0 & 0 & 1 & 0 \\
0 & 0 & 1 & 0 & 0 & 1
\end{bmatrix}.
\]
Substitution into \eqref{aa*+} yields
\begin{align*}
(\widehat A\widehat A^*)^\dagger
&=
(1-2\varepsilon)(U^\dagger)^*U^\dagger
+
(V^\dagger)^*V^\dagger\\&=[(U^\dagger)^*U^\dagger+(V^\dagger)^*V^\dagger]
+\varepsilon[-2(U^\dagger)^*U^\dagger]\\&:=G^\dagger+\varepsilon \Gamma_G,
\end{align*}
where
\[
G^\dagger
=
\begin{bmatrix}
\frac14&0&0&\frac14&0&0\\
0&\frac12&0&0&0&0\\
0&0&\frac12&0&0&0\\
\frac14&0&0&\frac14&0&0\\
0&0&0&0&\frac12&0\\
0&0&0&0&0&\frac12
\end{bmatrix},\qquad
\Gamma_G
=
\begin{bmatrix}
0&0&0&0&0&0\\
0&-\frac12&0&0&\frac12&0\\
0&0&-\frac12&0&0&\frac12\\
0&0&0&0&0&0\\
0&\frac12&0&0&-\frac12&0\\
0&0&\frac12&0&0&-\frac12
\end{bmatrix}.
\]
From \eqref{AAA}, it follows that
\begin{align}\label{Adagger}
\widehat A^\dagger
=\widehat A^*(\widehat A\widehat A^*)^\dagger=
\frac12
\begin{bmatrix}
0&0&0&0&0&0\\
0&0&-1&0&0&1\\
0&1&0&0&-1&0\\
1&0&0&1&0&0\\
0&1&0&0&1&0\\
0&0&1&0&0&1
\end{bmatrix}
+
\frac{\varepsilon}{2}
\begin{bmatrix}
0&0&0&0&0&0\\
0&0&1&0&0&-1\\
0&-1&0&0&1&0\\
0&0&0&0&0&0\\
0&0&0&0&0&0\\
0&0&0&0&0&0
\end{bmatrix}.
\end{align}

To examine motion recovery in this configuration, consider
\[
\boldsymbol\xi=
\begin{bmatrix}
1&2&3&4&5&6
\end{bmatrix}^{T}.
\]
Keeping the motion fixed, we generate the corresponding dual velocity as
\[
\widehat{\boldsymbol b}
=
\widehat A\boldsymbol\xi
=
\begin{bmatrix}
4&8&4&4&2&8
\end{bmatrix}^{T}
+
\varepsilon
\begin{bmatrix}
0&3&-2&0&-3&2
\end{bmatrix}^{T}.
\]
Using the expression for $\widehat A^\dagger$ given in \eqref{Adagger}, we obtain
the following motion estimate
\[
\widehat{\boldsymbol\xi}_{\mathrm{rec}}
=
\widehat A^\dagger\widehat{\boldsymbol b}
=
\begin{bmatrix}
0&2&3&4&5&6
\end{bmatrix}^{T}.
\]
This estimate reproduces the measurements exactly:
\[
\widehat A\widehat{\boldsymbol\xi}_{\mathrm{rec}}
=
\widehat{\boldsymbol b}.
\]

The first component cannot be determined from the measurements, since
\[
\widehat A
\begin{bmatrix}
1&0&0&0&0&0
\end{bmatrix}^{T}
=
O.
\]
Thus, every motion of the form
\[
\begin{bmatrix}
1+\alpha&2&3&4&5&6
\end{bmatrix}^{T},
\qquad \alpha\in\mathbb R,
\]
produces the same measurements. Rotation about the $x$-axis is therefore unobservable in this configuration. The zero first component of
$\widehat{\boldsymbol\xi}_{\mathrm{rec}}$ does not imply that the actual angular velocity about the $x$-axis vanishes.

\section*{Funding}
\begin{itemize}
\item Daochang Zhang is supported by the National Natural Science Foundation of China (NSFC) (No. 11901079), China Postdoctoral Science Foundation (No. 2021M700751), and the Scientific and Technological Research Program Foundation of Jilin Province (No. JJKH20190690KJ; No. JJKH20220091KJ; No. JJKH20250851KJ).
\end{itemize}

\section*{Conflict of Interest}
The authors declare that they have no potential conflict of interest.

\section*{Data Availability}
Data sharing not applicable to this article as no datasets were generated or analyzed during the current study.


\bibliographystyle{abbrv}

\end{document}